\documentclass[11pt]{article}
\usepackage{amsmath, amsthm, amssymb, amsfonts, latexsym,amscd}
\usepackage{tikz}
\usepackage{authblk}
\usetikzlibrary{shadings,patterns, decorations.pathreplacing}
\newtheorem{theorem}{Theorem}[section]
\newtheorem{corollary}[theorem]{Corollary}
\newtheorem{lemma}[theorem]{Lemma}
\newtheorem{proposition}[theorem]{Proposition}

\newtheorem{example}[theorem]{Example}

\newtheorem{conjecture}[theorem]{Conjecture}
\begin{document}

\title{\bf Total eccentricity index of graphs with fixed number of pendant or cut vertices}

\author[1,2]{Dinesh Pandey \thanks{Corresponding Author: dinesh.pandey@niser.ac.in, Supported by UGC Fellowship scheme (Sr. No. 2061641145), Government of India}}
\author[1,2]{Kamal Lochan Patra \thanks{klpatra@niser.ac.in}}
\affil[1]{School of Mathematical Sciences,
National Institute of Science Education and Research (NISER), Bhubaneswar,
P.O.- Jatni, District- Khurda, Odisha - 752050, India 
}

\affil[2]{Homi Bhabha National Institute (HBNI),
Training School Complex, Anushakti Nagar,
Mumbai - 400094, India}
\date{}
\maketitle

\begin{abstract}
The total eccentricity index of a connected graph is defined as sum of the eccentricities of all its vertices. In this paper, we give the sharp upper bound on the total eccentricity index over graphs with fixed number of pendant vertices and the sharp lower bound on the same over graphs with fixed number of cut vertices. We also provide the sharp upper bounds on the total eccentricity index over graphs with $s$ cut vertices for $s=0, 1, n-3, n-2$ and propose a conjecture for $2\leq s\leq n-4.$\\



\end{abstract}

\section{Introduction}
Throughout this article, graphs are finite, simple, connected and undirected. Let $G=(V,E)$ be a graph with vertex set $V(G)$ and edge set $E(G)$. The number of vertices adjacent to a vertex $v$ is called $\it{degree}$ of $v$ and we  denote it by $deg(v)$. A vertex of degree one in $G$ is called a {\it pendant vertex}. A vertex $w$ is called a {\it cut vertex} of $G$ if $G - w$ is disconnected. For two isomorphic graphs $G$ and $H$, we write $G\cong H$. A graph $G$ with $|V(G)|\geq 2$ is called a {\it $2$-connected} graph if it has no cut vertex.  For $u,v\in V(G)$, the distance $d_G(u,v)$ or $d(u,v)$ is the number of edges in a shortest path joining $u$ and $v$. For two subgraphs $H_1$ and $H_2$ of $G$, the distance $d(H_1,H_2)$ in $G$ is defined as $d(H_1,H_2)=\min\{d(u,v): u\in V(H_1), v\in V(H_2)\}$. The {\it eccentricity} of a vertex $v$, denoted by $e_G(v)$ or $e(v)$, is defined as $e(v)=\max\{d(v,u)|u\in V(G)\}$.  A vertex $u$ is called an {\it eccentric vertex} of $v$ if $e(v)=d(v,u)$. A vertex of minimum eccentricity in $G$ is called a {\it central  vertex} and the set of all central vertices of $G$ is called the {\it center} of $G$, denoted by $C(G)$. The {\it diameter}, $diam(G)$ of $G$ is the maximum eccentricity over all the vertices of $G$. The sum of distances between all unordered pair of vertices in a graph $G$ is called the {\it Wiener index} of $G$ and we denote it by $W(G)$. The {\it total eccentricity index} of  $G$ is defined as the sum of eccentricities of all its vertices and we denote it by $\varepsilon(G)$. The {\it average eccentricity} of $G$,  denoted by $avec(G),$ is defined as $avec(G)=\frac{\varepsilon(G)}{n}$, where $n$ is the order of $G$. \\

The average eccentricity is first defined by Buckley and Harary in \cite{Bh} (see Exercise 9.1, problem 4) by the name {\it eccentric  mean}. Different extremal problems related  to the average eccentricity are studied for various classes of graphs. In \cite{Do}, Dankelmann and Osaye gave an upper bound on average eccentricity of graphs with fixed maximum and minimum degree. In \cite{Tz2}, Tang and Zhou have given lower and upper bounds on the average eccentricity of trees with fixed diameter, fixed number of pendant vertices and fixed matching number. In \cite{I}, Ili\'c found the tree having maximum average eccentricity over trees with fixed maximum degree. In \cite{Dgs} Dankelmann et al. established an upper bound for average eccentricity over graphs with fixed minimum degree. The average eccentricity of unicylic graphs  and unicyclic graphs with fixed girth are studied in \cite{Tz1} and \cite{Yfw} respectively. Dankelmann and Mukwembi \cite{Dm} have given sharp upper bounds on the average eccentricity of graphs of order $n$ in terms of independence number, chromatic number, domination number or connected domination number.

Studying the extremal problems on the total eccentricity index in different classes of graphs is equivalent to the study of the same for the average eccentricity. The total eccentricity index is also studied by the name of eccentricity sum or total eccentricity (see \cite{Ssw} and \cite{Ffy}). Smith et al. \cite{Ssw}, determined the trees which maximise or minimise the total eccentricity index over trees with fixed degree sequence. In \cite{Dnp}, the total eccentricity index of generalised hierarchical product of graphs are studied. 

The Wiener index of different classes of graphs have been studied extensively. It has been observed that the Wiener index and the total eccentricity index have similar behaviour in various classes of graphs. For example, among all trees on $n$ vertices the Wiener index and the total eccentricity index is maximized by the path and minimized by the star. In many class of graphs, the graph which maximises (minimises) the total eccentricity index is same as the graph which maximises (minimises) the Wiener index. Although it is not necessary that for two graphs $G$ and $H$, if $W(G)<W(H)$ then $\varepsilon(G)< \varepsilon(H)$. In the following example, we have two graphs $G_1$ and $G_2$ with $W(G_1)<W(G_2)$ but $\varepsilon(G_1)> \varepsilon(G_2).$
\begin{example}
$W(G_1)=13, W(G_2)=14, \varepsilon(G_1)=10,\varepsilon(G_2)=9$
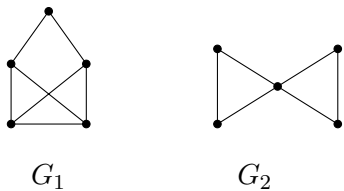
\begin{figure}[h]
\begin{center}
\begin{tikzpicture}[scale=1.0]
\filldraw (0,0) circle [radius=.5 mm];
\filldraw (1,0) circle [radius=.5 mm];
\filldraw (0,.8) circle [radius=.5mm];
\filldraw (1,.8) circle [radius=.5 mm];
\filldraw (0.5,1.5) circle [radius=.5 mm];
\draw (0,0)--(1,0)--(0,.8)--(0,0)--(1,.8)--(.5,1.5)--(0,0.8);
\draw (1,0)--(1,.8);
\draw (0.5,-.7) node {$G_1$};
\end{tikzpicture}
\hskip 1.5cm
\begin{tikzpicture}
\filldraw (0,0) circle [radius=.5 mm];
\filldraw (0,1) circle [radius=.5 mm];
\filldraw (.8,.5) circle [radius=.5 mm];
\filldraw (1.6,0) circle [radius=.5 mm];
\filldraw (1.6,1) circle [radius=.5 mm];
\draw (.8,.5)--(0,0)--(0,1)--(.8,.5)--(1.6,0)--(1.6,1)--(.8,.5);
\draw (0.5,-.7) node {$G_2$};
\end{tikzpicture}
\end{center}
\caption{Two graphs $G_1$ and $G_2$ with $W(G_1)<W(G_2)$ but $\varepsilon(G_1)>\varepsilon(G_2)$}
\end{figure}

\end{example}

In \cite{Dakd}, the authors have studied some relations between  total eccentricity index and Wiener index in graphs. In particular, they have  given some bounds on Wiener index in terms of total eccentricity index and also studied the difference $W(G)-\varepsilon(G)$ for some classes of graphs. In \cite{Pp}, Pandey and Patra established the sharp lower and upper bounds on the Wiener index over graphs with fixed number of pendant vertices and characterized  the graph which minimizes the Wiener index over graphs with fixed number of cut vertices. In this paper, we study the total eccentricity index over  graphs with fixed number of pendant vertices and graphs with fixed number of cut vertices. More specifically, We give the

\begin{itemize}
\item[i)]  sharp upper bounds on the total eccentricity index among all graphs on $n$ vertices with $k$ pendant vertices.
\item[ii)]  sharp lower bounds on the total eccentricity index among all graphs on $n$ vertices with $s$ cut vertices.
\item[iii)]  sharp upper bounds on the total eccentricity index among all graphs on $n$ vertices with $s$ cut vertices for $s=0,1,n-3$ and $n-2.$
 \end{itemize}
Some of these above bounds are achieved by a unique graph. In all the cases, we give a graph which attains the bound. We also propose a conjecture on the graph which attains the maximum total eccentricity index over graphs on $n$ vertices with  $2\leq s \leq n-4$ cut vertices.

\section{Preliminaries}

The following lemma  is very useful.
\begin{lemma}\label{edge}
Let $G$ be a graph and let $u,v\in V(G)$ be  two non-adjacent vertices. Let $G'$ be the graph obtained from $G$ by joining $u$ and $v$ with an edge. Then $\varepsilon(G)\geq \varepsilon(G').$ 
\end{lemma}
By Lemma \ref{edge}, it can be observed that among all graphs on $n$ vertices, the total eccentricity index is minimized by  $K_n$ and maximized by a tree. We have $\varepsilon(K_n)=n$.

Let  $v\in V(G)$ and $|V(G)|\geq 2.$ For $l,k \geq 1,$  let $G_{k,l}$ be the graph obtained from $G$ by attaching two new paths $P:vv_{1}v_{2}\cdots v_{k}$ and $Q:vu_{1}u_{2}\cdots u_{l}$  at $v$, where $u_{1},u_{2},\ldots,u_{l}$ and $v_{1},v_{2},\ldots,v_{k}$ are distinct new vertices. By $G_{0,l}$ we mean, attaching a path of length $l$ at $v$. Let ${\widetilde G}_{k,l}$ be the graph obtained from $G_{k,l}$ by removing the edge $\{v_{k-1},v_{k}\}$ and adding the edge $\{u_{l},v_{k}\} $. Observe that the graph ${\widetilde G}_{k,l}$ is isomorphic to the graph $G_{k-1,l+1}$. We say that ${\widetilde G}_{k,l}$ is obtained from $G_{k,l}$ by {\em grafting} an edge. 

\begin{lemma}[\cite{I}, Theorem 2.1 and \cite{Tz2}, Lemma 4.5] \label{Grafting}
If $1\leq k \leq l,$ then $\varepsilon(G_{k-1,l+1})>\varepsilon(G_{k,l}).$
\end{lemma}
In the following lemma, we compare the total eccentricity index of two graphs, where one is obtained from the other by some graph perturbation.
\begin{lemma}\label{Path-movement}
Let $H_1$ and $H_2$ be two connected graphs with $|V(H_1)|, |V(H_2)|\geq 2$ and $P=v_1v_2\ldots v_d$ be a path on $d\geq 2$ vertices. Suppose $u\in V(H_1)$ and $v\in V(H_2)$. Let $G$ be the graph obtained from $H_1, H_2$ and $P$ by identifying the vertices $u, v$ and $v_1.$ Let $G'$ be the graph obtained from $H_1, H_2$ and $P$ by identifying the vertices $u$ with $v_1$ and $v$ with $v_d$ (See Figure \ref{Fgr_pathmovement}). Then $\varepsilon(G')> \varepsilon(G).$ 
\end{lemma}

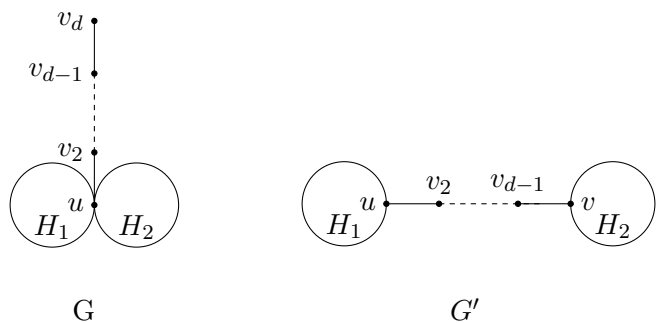
\begin{figure}[h!]
\begin{center}
\begin{tikzpicture}[scale =.7]
\draw (0.2,0) node [below]{$H_1$} circle [radius=.8cm]; 
\draw (1.8,0) node [below]{$H_2$} circle [radius=.8cm];
\draw (.8,-2) node {G};
\filldraw (1,0) node [left] {$u$} circle [radius=.5mm];
\filldraw (1,1) node [left] {$v_2$} circle [radius=.5mm];
\filldraw (1,2.5) node [left] {$v_{d-1}$} circle [radius=.5mm];
\filldraw (1,3.5) node [left] {$v_d$} circle [radius=.5mm];
\draw (1,0)--(1,1);
\draw (1,1) [dash pattern=on 2pt off 2pt]--(1,2.5);
\draw (1,2.5)--(1,3.5);
\end{tikzpicture}
\hskip 1.5 cm
\begin{tikzpicture}[scale= .7]
\draw (0.2,0) node [below]{$H_1$} circle [radius=.8cm];
\draw (5.3,0) node [below]{$H_2$} circle [radius=.8 cm];
\draw (2.5,-2) node {$G'$};
\filldraw(1,0) node [left] {$u$} circle [radius=.5mm];
\filldraw(2,0) node [above]{$v_2$} circle [radius=.5mm];
\filldraw(3.5,0) node[above]{$v_{d-1}$} circle [radius=.5mm];
\filldraw(4.5,0) node[right]{$v$} circle [radius=.5mm];
\draw (1,0)--(2,0);
\draw (2,0) [dash pattern= on 2pt off 2pt]--(4,0);
\draw(3.5,0)--(4.5,0);
\end{tikzpicture}
\end{center}
\caption{The graphs $G$ and $G'$}\label{Fgr_pathmovement}
\end{figure}

\begin{proof}
Let $x,y\in V(H_1)$. Then from the construction of both $G$ and $G'$ (see Figure \ref{Fgr_pathmovement}), it is clear that the length of all the shortest paths between $x$ and $y$ remain unchanged in both $G$ and $G'$. So, $d_{G'}(x,y)=d_G(x,y)$. Similarly, for any two vertices either in $H_2$ or  in $P$, the distance between them remain unchanged in both $G$ and $G'$. Now take one vertex in $H_i$, $i=1,2$ and the other vertex is in $P$. Without loss of generality, suppose $z\in V(H_1)$ and $w\in V(P)$. Then
\begin{align*}
d_{G'}(z,w)&=d_{G'}(z,u)+d_{G'}(u,w)\\
   &=d_G(z,u)+d_G(u,w)\\
   &=d_G(z,w).
\end{align*} 
Finally suppose $a\in V(H_1)$ and $b\in V(H_2)$. Then $d_{G'}(a,b)=d_{G'}(a,u)+d_{G'}(u,v)+d_{G'}(v,b) >d_{G'}(a,u)+d_{G'}(v,b)=d_G(a,u)+d_G(u,b)=d_G(a,b)$.
So, while moving from $G$ to $G'$, the eccentricity of each vertex either increases or remains same. To complete the proof, it is enough to show that there exist a vertex $x_0$ belongs to both $V(G)$ and $V(G')$ such that $e_{G'}(x_0)>e_{G}(x_0).$

Without loss of generality, take $diam(H_1)=D_1\geq diam(H_2)=D_2.$ Let $x_0\in V(H_2)$ be a vertex farthest from $v$. Then $d_{H_2}(x_0,v)\geq \frac{D_2}{2}.$ Similarly, there exists $y_0\in V(H_1)$ farthest from $u$ such that $d_{H_1}(y_0,u)\geq \frac{D_1}{2}.$ Then $d_G(x_0,y_0)=d_G(x_0,u)+d_G(u,y_0)\geq D_2$. So, an eccentric vertex of $x_0$ in $G$ lies outside $H_2$. Then $e_{G}(x_0)=d_G(x_0,v_d)$ or $e_{G}(x_0)=d_G(x_0,y_0).$ But $e_{G'}(x_0)=d_{G'}(x_0,v)+d_{G'}(v,u)+d_{G'}(u,y_0)>e_G(x_0).$ Hence $\varepsilon(G')>\varepsilon(G).$ 
\end{proof}
 
\begin{corollary} [\cite{Dgs}, Corollary 1 and \cite{Tz2}, Proposition 4.3]
Among all trees on $n$ vertices, the total eccentricity index  is maximised by the path and minimised by the star.
\end{corollary}
 An easy calculation gives $\varepsilon(K_{1,n-1})=2n-1$ and $\varepsilon(P_n)=\lfloor\frac{3n^2-2n}{4}\rfloor$. Since $\varepsilon(K_n)=n$, so for any connected graph $G$ with $n$ vertices, $$n\leq \varepsilon(G)\leq  \left\lfloor\frac{3n^2-2n}{4}\right\rfloor.$$
  
A $\it{block}$ in $G$ is a maximal $2$-connected subgraph of $G$. We say a block $B$ in $G$, a {\it pendant block} if there is exactly one cut vertex of $G$ in $B$. Two blocks in $G$ are said to be {\it adjacent blocks}  if they share a common cut vertex. Let $B_G$ be the block graph corresponding to $G$ with $V(B_G)$ as the set of blocks of $G$ and two vertices $u$ and $v$ in $B_G$ are adjacent whenever the  corresponding blocks contains a common cut vertex of $G$. A block corresponding to a central vertex in $B_G$  is called a {\it central block} of G.  
  
\begin{lemma}
Let $B$ and $B'$ be two blocks in a graph $G$. Suppose  $d(B,B')$ is maximum among all pairs of blocks in $G$. Then both $B$ and $B'$ are pendant blocks.
\end{lemma}
\begin{proof}
Suppose $B$ is not a pendant block. Then $B$ contains at least two cut vertices of $G$. Let $u$ and $v$ be two cut vertices of $G$ in $B$ such that $d(B,B')=d(\{u\},B')$. Since $v$ is a cut vertex of $G$ there exists a block $B''$ (other than $B$ and $B'$) containing $v$. Then $d(B'',B')=d(\{v\},B')=d(v,u)+d(\{u\},B')>d(B,B'),$ which is a contradiction. Hence $B$ is a pendant block of $G$. Similarly, we can show that $B'$ is also a pendant block of $G$.
\end{proof}
  
\begin{lemma}\label{Block-transform}
Let $B$ be a block of a  graph $G$. Suppose $|V(B)|=r\geq 3$ and at most one vertex of $B$ is a cut vertex in $G$. Let $G'$ be the graph obtained  from $G$ by replacing the block $B$ with  the cycle $C_r$ such that the cut vertices (if any) of $G$ and $G'$ remain same. Then $\varepsilon(G')\geq \varepsilon(G).$  
\end{lemma}
\begin{proof}
Since $|V(B)|\geq 3$ and $B$ is a block in $G$, so $B$ must contain a cycle. Let $x\in V(G)$ and  $y$ be an eccentric vertex of $x$ in $G$. 

First suppose $G$ has no cut vertex. Then $G=B$ and $G'=C_r.$ Let $C_l:v_1v_2\ldots v_lv_1$ be a largest cycle in $B$.  Since $G$ has no cut vertex, so there exists a cycle $C_k$ containing both $x$ and $y$. As $C_l$ is a largest cycle of $G$, so $k\leq l$. Therefore, $e_G(x)\leq \lfloor \frac{k}{2} \rfloor \leq \lfloor \frac{l}{2} \rfloor \leq \lfloor \frac{r}{2} \rfloor.$ But the eccentricity of any vertex of $G'$ is $\lfloor \frac{r}{2} \rfloor$. Hence the result follows for this case. 
  
Now suppose $w\in V(B)$  is the cut vertex in $G.$  Let $C_m:wv_2\ldots v_m w$ be a largest cycle in $B$ containing $w$. Delete the vertices of $B$ not in $C_m$ and insert those $r-m$ vertices between  $v_{\lceil \frac{m}{2} \rceil}$ and $v_{\lceil \frac{m}{2} \rceil+1}$ to form $G'$ from $G$. Let $S=(V(G)\setminus V(B))\cup\{w\}\subseteq V(G)$. Since $(V(G)\setminus V(B))\cup\{w\}=(V(G')\setminus V(C_r))\cup\{w\}$ so $S\subseteq V(G')$.

 First suppose $x\in S\subseteq V(G)$. If $y\in S$  then $e_{G}(x)=d_{G}(x,y)=d_{G'}(x,y)\leq e_{G'}(x).$ If $y\in V(G)\setminus S$, then $e_G(x)=d_{G}(x,y)=d_G(x,w)+d_G(w,y)\leq d_{G'}(x,w)+\lfloor \frac{r}{2} \rfloor =e_{G'}(x).$\\

 Now suppose $x\in V(G)\setminus S= V(G')\setminus S.$ If $y\in V(G)\setminus S$, $e_{G}(x)=d_{G}(x,y)\leq \lfloor \frac{r}{2} \rfloor\leq e_{G'}(x).$ If $y\in S$, then $e_G(x)=d_G(x,y)=d_G(x,w)+d_G(w,y).$ In $G'$, if $x\in \{v_2,\ldots, v_m\}$ then $d_G(x,w)\leq  d_{G'}(x,w)$ and hence $e_G(x)=d_G(x,w)+d_G(w,y)\leq d_{G'}(x,w)+d_{G'}(w,y)=d_{G'}(x,y)\leq e_{G'}(x)$. In $G'$, if $x\notin \{v_2,\ldots, v_m\}$ then $d_{G'}(x,w)\geq \lfloor \frac{m}{2} \rfloor$. Since $C_m$ is the largest cycle in $B$ containing $w$, so $d_{G}(x,w)\leq \lfloor \frac{m}{2} \rfloor\leq d_{G'}(x,w)$. Therefore, $e_G(x)=d_G(x,w)+d_G(w,y)\leq d_{G'}(x,w)+d_{G'}(w,y)=d_{G'}(x,y)\leq e_{G'}(x)$. This completes the proof. 
\end{proof}

\section{Graphs with fixed number of pendant vertices}
  
We denote the set of all connected graphs of order $n$ with $k$ pendant vertices by $\mathfrak{H}_{n,k}$. By $\mathfrak{T}_{n,k}$, we denote  the set of all trees on $n$ vertices with $k$ pendant vertices. Clearly  $\mathfrak{T}_{n,k}\subseteq \mathfrak{H}_{n,k}$. Note that a graph $G\in \mathfrak{H}_{n,k}$ has $n$ pendant vertices if and only if $G\cong K_2,$ also $K_{1,n-1}$ is the only element of $\mathfrak{H}_{n,n-1}$. So, we consider $0\leq k\leq n-2$. 
Let $T(l,m,d)$ be the tree obtained by attaching $l$ pendant vertices at one end and $m$ pendant vertices to the another end of the path $P_d.$ So, $ T(l,m,d)\in \mathfrak{T}_{l+m+d,l+m}$. Note that all the elements  of $\mathfrak{H}_{n,n-2}$ are of the form $T_{l,n-l-2,2}$, where $1\leq l\leq n-3$. It is easy to check that eccentricity of $T_{l,n-l-2,2}$ is same for any $l$, $1\leq l\leq n-3$. So to study the total eccentricity index over $\mathfrak{H}_{n,k}$, we consider $0\leq k\leq n-3.$

By $sG$ we mean, the graph consisting of $s$ copies of $G.$ Let $T_{n,k} \in \mathfrak{T}_{n,k} $ be the tree that has a vertex $v$ of degree $k$ and $T_{n,k} - v = r P_{q+1} \cup (k-r) P_q$, where $q=\lfloor\frac{n-1}{k} \rfloor$ and $r = n-1 -kq$. Here, we have $0\leq r <k.$

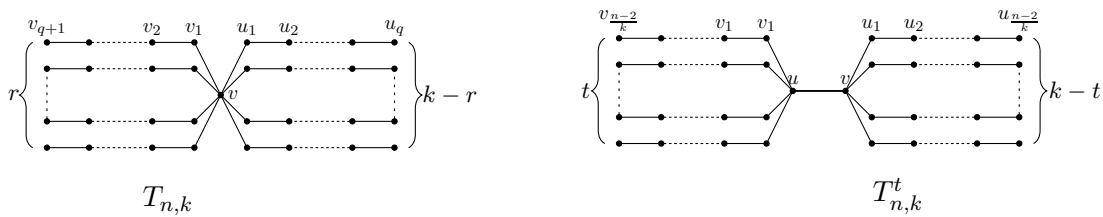
\begin{figure}[h]
\begin{center}
\begin{tikzpicture}[scale=.7]
\filldraw node[scale=.7](0,0)[right]{$v$}circle [radius=0.5mm];

\filldraw (.5,1)node[scale=.7][above]{$u_{1}$}circle [radius=0.5mm];
\filldraw (0.5,0.5)circle [radius=.5mm];
\filldraw (0.5,-0.5)circle [radius=.5mm];
\filldraw (0.5,-1)circle [radius=.5mm];
\filldraw (-0.5,1)node[scale=.7][above]{$v_{1}$}circle [radius=.5mm];
\filldraw (-0.5,-0.5)circle [radius=.5mm];
\filldraw (-0.5,-1)circle [radius=.5mm];
\filldraw (-0.5,0.5)circle [radius=.5mm];

\filldraw (1.3,1)node[scale=.7][above]{$u_{2}$}circle [radius=0.5mm];
\filldraw (1.3,0.5)circle [radius=.5mm];
\filldraw (1.3,-0.5)circle [radius=.5mm];
\filldraw (1.3,-1)circle [radius=.5mm];
\filldraw (-1.3,1)node[scale=.7][above]{$v_{2}$}circle [radius=.5mm];
\filldraw (-1.3,-0.5)circle [radius=.5mm];
\filldraw (-1.3,-1)circle [radius=.5mm];
\filldraw (-1.3,0.5)circle [radius=.5mm];

\filldraw (2.5,1)circle [radius=0.5mm];
\filldraw (2.5,0.5)circle [radius=.5mm];
\filldraw (2.5,-0.5)circle [radius=.5mm];
\filldraw (2.5,-1)circle [radius=.5mm];
\filldraw (-2.5,1) circle[radius=.5mm];
\filldraw (-2.5,-0.5)circle [radius=.5mm];
\filldraw (-2.5,-1)circle [radius=.5mm];
\filldraw (-2.5,0.5)circle [radius=.5mm];

\filldraw (3.3,1)node[scale=.7][above]{$u_{q}$} circle[radius=0.5mm];
\filldraw (3.3,0.5)circle [radius=.5mm];
\filldraw (3.3,-0.5)circle [radius=.5mm];
\filldraw (3.3,-1)circle [radius=.5mm];
\filldraw (-3.3,1)node[scale=.7][above]{$v_{q+1}$}circle [radius=.5mm];
\filldraw (-3.3,-0.5)circle [radius=.5mm];
\filldraw (-3.3,-1)circle [radius=.5mm];
\filldraw (-3.3,0.5)circle [radius=.5mm];

\draw (-1.3,1)--(-0.5,1)--(0,0)--(0.5,1)--(1.3,1);
\draw (-1.3,.5)--(-0.5,.5)--(0,0)--(0.5,.5)--(1.3,.5);
\draw (-1.3,-1)--(-0.5,-1)--(0,0)--(0.5,-1)--(1.3,-1);
\draw (-1.3,-.5)--(-0.5,-.5)--(0,0)--(0.5,-.5)--(1.3,-.5);

\draw (3.3,1)--(2.5,1);
\draw (3.3,.5)--(2.5,.5);
\draw (3.3,-1)--(2.5,-1);
\draw (3.3,-.5)--(2.5,-.5);
\draw (-3.3,1)--(-2.5,1);
\draw (-3.3,.5)--(-2.5,.5);
\draw (-3.3,-1)--(-2.5,-1);
\draw (-3.3,-.5)--(-2.5,-.5);

\draw [dash pattern= on 1pt off 1pt](1.3,1)--(2.5,1);
\draw [dash pattern= on 1pt off 1pt](1.3,.5)--(2.5,.5);
\draw [dash pattern= on 1pt off 1pt](1.3,-1)--(2.5,-1);
\draw [dash pattern= on 1pt off 1pt](1.3,-.5)--(2.5,-.5);
\draw [dash pattern= on 1pt off 1pt](-1.3,1)--(-2.5,1);
\draw [dash pattern= on 1pt off 1pt](-1.3,.5)--(-2.5,.5);
\draw [dash pattern= on 1pt off 1pt](-1.3,-1)--(-2.5,-1);
\draw [dash pattern= on 1pt off 1pt](-1.3,-.5)--(-2.5,-.5);

\draw [dash pattern= on 1pt off 2pt](3.3,.5)--(3.3,-.5);
\draw [dash pattern= on 1pt off 2pt](-3.3,.5)--(-3.3,-.5);

\draw [decorate,decoration={brace,amplitude=5pt},xshift=-2pt,yshift=0pt]
(-3.5,-1) -- (-3.5,1) node [black,midway,xshift=-0.25cm] 
{\footnotesize $r$};

\draw [decorate,decoration={brace,amplitude=5pt, mirror},xshift=2pt,yshift=0pt]
(3.5,-1) -- (3.5,1) node [black,midway,xshift=0.55cm] 
{\footnotesize $k-r$};
\draw(-1,-2) node{$T_{n,k}$};
\end{tikzpicture}
\hskip 1cm
\begin{tikzpicture} [scale=.7]
\filldraw (0,0)node[scale=.7][above]{$v$}circle [radius=0.5mm];
\filldraw (-1,0)node[scale=.7][above]{$u$}circle [radius=0.5mm];

\filldraw (0.5,1)node[scale=.7][above]{$u_{1}$}circle [radius=0.5mm];
\filldraw (0.5,0.5)circle [radius=.5mm];
\filldraw (0.5,-0.5)circle [radius=.5mm];
\filldraw (0.5,-1)circle [radius=.5mm];
\filldraw (-1.5,1)node[scale=.7][above]{$v_{1}$}circle [radius=.5mm];
\filldraw (-1.5,-0.5)circle [radius=.5mm];
\filldraw (-1.5,-1)circle [radius=.5mm];
\filldraw (-1.5,0.5)circle [radius=.5mm];

\filldraw (1.3,1)node[scale=.7][above]{$u_{2}$}circle [radius=0.5mm];
\filldraw (1.3,0.5)circle [radius=.5mm];
\filldraw (1.3,-0.5)circle [radius=.5mm];
\filldraw (1.3,-1)circle [radius=.5mm];
\filldraw (-2.3,1)node[scale=.7][above]{$v_{1}$}circle [radius=.5mm];
\filldraw (-2.3,-0.5)circle [radius=.5mm];
\filldraw (-2.3,-1)circle [radius=.5mm];
\filldraw (-2.3,0.5)circle [radius=.5mm];

\filldraw (2.5,1)circle [radius=0.5mm];
\filldraw (2.5,0.5)circle [radius=.5mm];
\filldraw (2.5,-0.5)circle [radius=.5mm];
\filldraw (2.5,-1)circle [radius=.5mm];
\filldraw (-3.5,1)circle [radius=.5mm];
\filldraw (-3.5,-0.5)circle [radius=.5mm];
\filldraw (-3.5,-1)circle [radius=.5mm];
\filldraw (-3.5,0.5)circle [radius=.5mm];

\filldraw (3.3,1)node[scale=.7][above ]{$u_{\frac{n-2}{k}}$}circle [radius=0.5mm];
\filldraw (3.3,0.5)circle [radius=.5mm];
\filldraw (3.3,-0.5)circle [radius=.5mm];
\filldraw (3.3,-1)circle [radius=.5mm];
\filldraw (-4.3,1)node[scale=.7][above]{$v_{\frac{n-2}{k}}$}circle [radius=.5mm];
\filldraw (-4.3,-0.5)circle [radius=.5mm];
\filldraw (-4.3,-1)circle [radius=.5mm];
\filldraw (-4.3,0.5)circle [radius=.5mm];

\draw (-2.3,1)--(-1.5,1)--(-1,0)-- (0,0)--(0.5,1)--(1.3,1);
\draw (-2.3,.5)--(-1.5,.5)--(-1,0)--(0,0)--(0.5,.5)--(1.3,.5);
\draw (-2.3,-1)--(-1.5,-1)--(-1,0)--(0,0)--(0.5,-1)--(1.3,-1);
\draw (-2.3,-.5)--(-1.5,-.5)--(-1,0)--(0,0)--(0.5,-.5)--(1.3,-.5);
\draw (3.3,1)--(2.5,1);
\draw (3.3,.5)--(2.5,.5);
\draw (3.3,-1)--(2.5,-1);
\draw (3.3,-.5)--(2.5,-.5);

\draw (-3.5,1)--(-4.3,1);
\draw (-3.5,.5)--(-4.3,.5);
\draw (-3.5,-1)--(-4.3,-1);
\draw (-3.5,-.5)--(-4.3,-.5);

\draw [dash pattern= on 1pt off 1pt](1.3,1)--(2.5,1);
\draw [dash pattern= on 1pt off 1pt](1.3,.5)--(2.5,.5);
\draw [dash pattern= on 1pt off 1pt](1.3,-1)--(2.5,-1);
\draw [dash pattern= on 1pt off 1pt](1.3,-.5)--(2.5,-.5);
\draw [dash pattern= on 1pt off 1pt](-2.3,1)--(-3.5,1);
\draw [dash pattern= on 1pt off 1pt](-2.3,.5)--(-3.5,.5);
\draw [dash pattern= on 1pt off 1pt](-2.3,-1)--(-3.5,-1);
\draw [dash pattern= on 1pt off 1pt](-2.3,-.5)--(-3.5,-.5);

\draw [dash pattern= on 1pt off 2pt](3.3,.5)--(3.3,-.5);
\draw [dash pattern= on 1pt off 2pt](-4.3,.5)--(-4.3,-.5);

\draw [decorate,decoration={brace,amplitude=5pt},xshift=-2pt,yshift=0pt]
(-4.5,-1) -- (-4.5,1) node [black,midway,xshift=-0.25cm] 
{\footnotesize $t$};

\draw [decorate,decoration={brace,amplitude=5pt, mirror},xshift=2pt,yshift=0pt]
(3.5,-1) -- (3.5,1) node [black,midway,xshift=0.55cm] 
{\footnotesize $k-t$};

\draw(1,-2) node{$T_{n,k}^t$};
\end{tikzpicture}
\caption{The graphs $T_{n,k}$ and $T_{n,k}^t$}
\end{center}
\end{figure}

For $k|n-2,$ let $T_{n,k}^t \in \mathfrak{T}_{n,k}$ be the tree that has two adjacent vertices $u$ and $v$ with degrees $t+1$ and $k-t+1$, respectively, such that $T_{n,k}^t -u-v=kP_{\frac{n-2}{k}}.$ 

In \cite{Tz2}, the authors have studied the total eccentricity index (in terms of average eccentricity) of trees over $\mathfrak{T}_{n,k}$ and proved the following results.
  
\begin{proposition}[\cite{Tz2}, Proposition 4.7] \label{k-maxtree}
Let $T\in \mathfrak{T}_{n,k}.$ Then

$$\varepsilon(T)\leq \varepsilon(T(l,k-l,n-k)),$$ for any $1\leq l \leq k-1,$ and
$$\varepsilon(T)\geq \begin{cases}
\varepsilon(T_{n,k}) &\mbox{if}\;\; k\nmid n-2,\\
\varepsilon(T_{n,k}^t) &\mbox{if}\;\; k\mid n-2.
\end{cases}$$
for any $1\leq t \leq k-1.$
\end{proposition}
Since $\varepsilon(P_n)$ is known, the value of $\varepsilon(T(l, m,d))$ can be easily calculated. We have
$$\varepsilon(T(l,k-l,n-k))=\left\lfloor\dfrac{3n^2-k^2-2nk+2(n+k)}{4}\right\rfloor.$$
The values of $\varepsilon(T_{n,k})$ and $\varepsilon(T_{n,k}^t)$ are given in \cite{Tz2}.

For $1\leq g \leq n-1$, the  graph obtained by joining one end vertex of the path $P_{n-g}$ with a vertex of $C_g$ by an edge is a unicyclic graph on $n$ vertices and we denote it by $U_{n,g}^l$. By $U_{n,g}^p$, we denote the unicyclic graph obtained by adding $n-g$ pendant vertices to a vertex of $C_g$ (see Figure \ref{ung}). Note that  $U_{n,n-1}^l\cong U_{n,n-1}^p$.
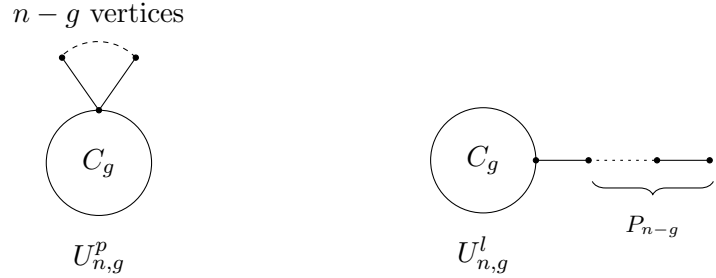
\begin{figure}[h!]
\begin{center}
\begin{tikzpicture}[scale =.7]
\draw (0,0) circle [radius=1 cm];
\filldraw (0,1) circle [radius= .5mm];
\filldraw (-.7,2) circle [radius= .5mm];
\filldraw (.7,2) circle [radius= .5mm];
\draw (-.7,2)--(0,1)--(.7,2);
 \draw [dash pattern=on 2pt off 2pt](.7,2) arc [ start angle=44, end angle =135, radius=1cm];
 \draw (0,2.8) node {$n-g$ vertices};
 \draw (0,0) node {$C_g$};
 \draw ( 0,-1.8) node {$U_{n,g}^p$} ;
\end{tikzpicture}
\hskip 3cm
\begin{tikzpicture}[scale=.7]
\draw (0,0) circle [radius= 1cm];
\filldraw (1,0) circle [radius=.5mm];
\filldraw (2,0) circle [radius= .5mm];
\filldraw (3.3,0) circle [radius=.5mm];
\filldraw (4.3,0) circle [radius= .5mm];
\draw [dash pattern=on 1pt off 2pt] (2,0)--(3.3,0);
\draw (1,0)--(2,0);
\draw(3.3,0)--(4.3,0);
\draw(0,0) node {$C_g$};
\draw [decorate, decoration={brace, amplitude=5pt, mirror}, xshift =2pt, yshift=0pt]
(2,-0.4) -- (4.3,-0.4) node [black, midway, yshift=-0.6cm]
{\footnotesize $P_{n-g}$};
\draw (0,-1.8) node {$U_{n,g}^l$} ;
\end{tikzpicture}
\caption{The graphs $U_{n,g}^p$ and $U_{n,g}^l$}\label{ung}
\end{center}
\end{figure}

In \cite{Tz1} and \cite{Yfw}, the authors have studied the total eccentricity index (in terms of average eccentricity) of unicyclic graphs and proved the following.
\begin{proposition}[\cite{Tz1}, Theorem 3.3 and \cite{Yfw}, Theorem 2.3]\label{Uc-max}
Let $H$ be a unicyclic graph on $n\geq 5$ vertices. Then
$$2n-1\leq \varepsilon(H)\leq \left\lfloor \dfrac{3n^2-4n-3}{4}\right\rfloor.$$ Furthermore, the left equality happens if and only if $H\cong U_{n,3}^p$ and the right equality happens if and only if $H\cong U_{n,3}^l$.
\end{proposition}
The authors have proved the Proposition \ref{Uc-max} for $n\geq 6$. A simple  calculation shows that the result is also true for $n=5$. So, we modified the statement in Proposition \ref{Uc-max} by taking $n\geq 5.$ 

For $0\leq k \leq n-3$, let $P_n^k$ be the graph in $\mathfrak{H}_{n,}k$  obtained by attaching $k$ pendant vertices at one vertex of $K_{n-k}.$ Clearly the complete graph $K_n$ uniquely minimizes the total eccentricity index over $\mathfrak{H}_{n,0}$.   For $1\leq k \leq n-3$, we have the following result due to Tang and West.
 
\begin{theorem}[\cite{Tw}, Theorem $3.4$]\label{min-pendant}
Let  $G\in \mathfrak{H}_{n,k}$, $n\geq 3$ and $1\leq k\leq n-3,$ . Then $\varepsilon(G)\geq 2n-1$ and the equality happens if  $G\cong P_n^k.$
\end{theorem}  
 We now discuss the maximization problem on the total eccentricity index over $\mathfrak{H}_{n,k}$. For that the following lemma is very useful.
\begin{lemma}\label{g-3change}
Let $H$ be a connected graph with at least two vertices and $u\in V(H)$. Let $G$ be the graph obtained  by joining an edge between $u$ and the pendant vertex of $U_{r,g}^l$ with $g\geq 4$. Let $G'$ be the graph obtained  by joining an edge between $u$ and  the pendant vertex of $U_{r,3}^l$.  Then $\varepsilon(G)<\varepsilon(G').$
\end{lemma}
\begin{proof}
Let $V(U_{r,g}^l)=\{v_1,v_2,\ldots,v_r\}$ where $v_1v_2\cdots v_{r-g}$ is the path $P_{r-g}$, $v_{r-g+1}v_{r-g+2}\cdots v_{r}v_{r-g+1}$ is the cycle $C_g$ and $v_{r-g}$ is adjacent to $v_{r-g+1}.$  We form $G'$ from $G$ by naming the vertices of $U_{r,3}^l$ as following: $v_1v_2\cdots v_{r-3}$ is the path $P_{r-3}$, $v_{r-2}v_{r-1}v_{r}v_{r-2}$ is the cycle $C_3$ and $v_{r-3}$ is adjacent to $v_{r-2}$. Note that $V(G)=V(G').$ We will prove that for any $v\in V(G)=V(G')$, $e_{G}(v)\leq e_{G'}(v)$ and strict inequality occurs for at least one vertex.

Suppose $v\in V(G)=V(G')$. Let $v'$ be an eccentric vertex of $v$ in $G$. Following are the three cases.

\noindent{\bf Case-I:} $v\in V(H)$.\\
If $v'\in V(H)$, then $e_G(v)=d_G(v,v')=d_{G'}(v,v')\leq e_{G'}(v)$. If $v'\notin  V(H)$, then $v_{r-1}$ is an eccentric vertex of $v$ in $G'$. So, $e_G(v)=d_G(v,v')=d_G(v,u)+d_G(u,v')\leq d_{G'}(v,u)+r-1=e_{G'}(v).$

\noindent{\bf Case-II:} $v\in \{v_1,v_2,\cdots,v_{r-1}\}$.\\
We have $d_{G}(v_i,u)\leq d_{G'}(v_i,u)$ for $1\leq i\leq r-1$. If $v'\in V(H)$, then $e_G(v)=d_G(v,v')=d_G(v,u)+d_G(u,v')\leq d_{G'}(v,u)+d_{G'}(u,v')=d_{G'}(v,v')\leq e_{G'}(v)$. If $v'\notin  V(H)$ then $v'$ is in the cycle $C_g$. So, 
$$e_G(v_i)=\begin{cases} 
r-g-i+1+\lfloor\frac{g}{2}\rfloor &\mbox{if}\;\; 1\leq i\leq r-g,\\
\lfloor\frac{g}{2}\rfloor & \mbox{if} \;\; r-g+1\leq i\leq r-1. 
\end{cases}$$
Thus, for $1\leq i\leq r-g$, $v_{r-1}$ is an eccentric vertex of $v_i$ in $G'$. So, we have $e_G(v_i)= r-g-i+1+\lfloor\frac{g}{2}\rfloor= r-i-(g-\lfloor\frac{g}{2}\rfloor-1)\leq r-i-1=e_{G'}(v_i)$. The last inequality holds since $g\geq 4$. For $r-g+1\leq i\leq r-1$, $e_G(v_i)=\lfloor\frac{g}{2}\rfloor\leq e_{G'}(v_i)$, since $v_i$ lies on a path in $G'$ of length at least $g$.

\noindent{\bf Case-III:} $v=v_r.$\\
Let $z$ be a vertex of $H$ farthest from $u$. In this case, $e_{G}(v_r)=\max\{r-g+2+d(u,z), \lfloor\frac{g}{2}\rfloor \}$. Since  $g\geq 4$ so $r-1>\max\{r-g+2,\lfloor\frac{g}{2}\rfloor\}$ and hence  $e_{G}(v_r)<r-1+d(u,z)=e_{G'}(v_r).$ 

Therefore, $\varepsilon(G)<\varepsilon(G')$ and this completes the proof.                                                                                                                                                                                                                                      
\end{proof}

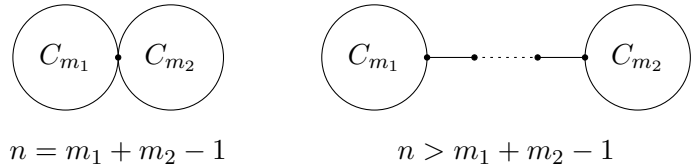
\begin{figure}[h!]
\begin{center}
\begin{tikzpicture}[scale =.7]
\draw (0,0)circle [radius= 1cm];
\draw (2,0) circle [radius= 1cm];
\draw  (1,-1.8) node {$n=m_1+m_2-1$};
\draw (0,0) node {$C_{m_1}$};
\draw (2,0) node {$C_{m_2}$};
\filldraw(1,0) circle [radius=.5mm];
\end{tikzpicture}
\hskip 1cm
\begin{tikzpicture}[scale=.7]
\draw (0,0)circle [radius= 1cm];
\draw (5,0) circle [radius= 1cm];
\filldraw (1,0)circle [radius= .5 mm];
\filldraw (1.9,0) circle [radius= .5 mm];
\filldraw (3.1,0)circle [radius= .5 mm];
\filldraw (4,0) circle [radius= .5 mm];
\draw [dash pattern= on 1pt off 2pt] (1.9,0)--(3.1,0);
\draw (1,0)--(1.9,0);
\draw(3.1,0)--(4,0);
\draw  (2.5,-1.8) node {$n>m_1+m_2-1$};
\draw (0,0) node {$C_{m_1}$};
\draw (5,0) node {$C_{m_2}$};
\end{tikzpicture}
\caption{The graphs $C_{m_1,m_2}^n$}\label{cmn}
\end{center}
\end{figure}
We define a specific subclass of graphs in $ \mathfrak{H}_{n,0} $ as follows. Let $m_1,m_2$ and  $n$ be positive integers with $m_1,m_2\geq 3$ and $n \geq m_1+m_2-1.$ If $n > m_1+m_2-1,$ take a path on $n-(m_1+m_2)+2$ vertices and identify one pendant vertex of the path with a vertex of $C_{m_1}$ and another pendant vertex with a vertex of $C_{m_2}.$ If $n= m_1+m_2-1$, then identify one vertex of $C_{m_1}$ with a vertex of $C_{m_2}.$ We denote this graph by $C_{m_1,m_2}^n$ (see Figure \ref{cmn}). 
  
\begin{lemma}\label{Compare}
For $n\geq 7$, $\varepsilon(C_{3,3}^n)>\varepsilon(C_n)$.
\end{lemma}
\begin{proof}
We have $$\varepsilon(P_n)=\begin{cases}
\frac{3n^2-2n}{4} &\mbox{if $n$ is even},\\
\frac{3n^2-2n-1}{4} &\mbox{if $n$ is odd}. 
\end{cases}$$
and it is easy to check that 
  $$\varepsilon(C_n)=\begin{cases}
\frac{n^2}{2} &\mbox{if $n$ is even},\\
\frac{n^2-n}{2} &\mbox{if $n$ is odd}. 
\end{cases}$$
Also  $\varepsilon(C_{3,3}^n)=\varepsilon(P_{n-2})+2(n-3)$. So,
$$\varepsilon(C_{3,3}^n)=
\begin{cases}
\frac{3}{4}n^2-\frac{3}{2}n-2 &\mbox{if $n$ is even},\\
\frac{3}{4}n^2-\frac{3}{2}n-\frac{9}{4} &\mbox{if $n$ is odd}. 
\end{cases}$$
Hence 
\begin{align*}\varepsilon(C_{3,3}^n)-\varepsilon(C_n)&=
\begin{cases}
\frac{n^2}{4}-\frac{3}{2}n-2 \; \mbox{if $n$ is even},\\
\frac{n^2}{4}-n-\frac{9}{4} \;  \mbox{if $n$ is odd}, 
\end{cases}\\
&>0, \; \mbox{for}\;  n\geq 7.
\end{align*}
\end{proof}

\begin{lemma}\label{Compare1}
Let $m_1,m_2\geq 3$ and $n=m_1+m_2-1$. For $n\geq 7$, $ \varepsilon(C_{3,3}^n)>\varepsilon(C_{m_1,m_2}^n)$.
\end{lemma}
\begin{proof}
Without loss of generality assume that $m_1\leq m_2.$ As $n\geq 7,$ so $m_2\geq 4.$ If $m_2=4$, then $C_{m_1,m_2}^n=C_{4,4}^7$ and  $\varepsilon(C_{4,4}^7)= 22< 24= \varepsilon(C_{3,3}^n)$. 

Suppose $n\geq 8.$ Then $m_2\geq 5.$ Let $C_{m_2}:v_1v_2\cdots v_{m_2}v_1$  be the subgraph of $C_{m_1,m_2}^n$ with $\deg(v_1)=4.$  Delete the edge $\{v_{\lceil \frac{m_2}{2}\rceil},v_{\lceil \frac{m_2}{2}\rceil+1} \}$ from $C_{m_1,m_2}^n$ to get a new graph $G$. By Lemma \ref{edge}, $\varepsilon(C_{m_1,m_2}^n)\leq \varepsilon(G).$ Note that in $G$ there are two paths ($v_1,v_2,\ldots, v_{\lceil \frac{m_2}{2}\rceil} $ and $v_1,v_{m_2}, \ldots v_{\lceil \frac{m_2}{2}\rceil+1}$) attached at $v_1$ each of length at least $2.$ By grafting of edges operation we can get a new graph $G'$ from $G$ where $G'$ is the graph in which two paths, one of length $1$ and other of length $m_2-2$ are attached at a vertex $v_1$ of $C_{m_1}.$  Then by Lemma \ref{Grafting}, $\varepsilon(G)<\varepsilon(G').$ 

Let $P:v_1v_2$ and $P':v_1v_{m_2}v_{m_2-1}\cdots v_3$ be the paths attached at $v_1$ in $G'$. Observe that if $v_2$ is an eccentric vertex of some $x\in V(G')$ then $m_1=3$ and the two vertices of $C_3$ other than $v_1$ are also the eccentric vertices of $x.$ So, we will not consider $v_2$ as an eccentric vertex of any vertex of $G'$. Construct a new graph $G''$ from $G'$ by deleting the edge $\{v_1,v_2\}$ and adding the edges $\{v_2,v_4\}$ and $\{v_2,v_3\}$. 

For any $x\in V(G')\setminus \{v_2\}$, $e_{G'}(x)=e_{G''}(x)$. Let $w\in V(C_{m_1})$ be a vertex farthest from $v_1$. Then $e_{G''}(v_2)=d_{G''}(v_2,v_1)+d_{G''}(v_1,w)=m_2-2+d_{G''}(v_1,w)= m_2-2+d_{G'}(v_1,w)\geq \max\{m_2-1,1+d_{G'}(v_1,w)\}=e_{G'}(v_2)$ and hence $\varepsilon(C_{m_1,m_2}^n)<\varepsilon(G')\leq \varepsilon(G'')$. Note that $G''\cong C_{m_1,3}^n$. If $m_1\geq 4$, then  the result follows from Lemma \ref{g-3change}.
\end{proof}

We will now prove the main results of this section.  
\begin{theorem}
Let $G\in \mathfrak{H}_{n,k}$ and $0\leq k \leq n-3$. Then
\begin{itemize}
\item[$(i)$] for $2\leq k \leq n-3$,  $\varepsilon(G)\leq \left\lfloor\dfrac{3n^2-k^2-2nk+2(n+k)}{4}\right\rfloor$ and equality is attained by the trees $T(l,k-l,n-k)$ for any $1\leq l\leq k-1$.
\item[$(ii)$] for $k=1$, $\varepsilon(G)\leq \left\lfloor \frac{3n^2-4n-3}{4} \right\rfloor $ and equality holds if and only if $G\cong U_{n,3}^l.$
\item[$(iii)$] For $n\geq 7$ and $k=0,$ $\varepsilon(G) \leq \left\lfloor \frac{3n^2-6n-8}{4} \right\rfloor$ and equality holds if and only if $G\cong C_{3,3}^n.$
\end{itemize}
\end{theorem}
\begin{proof}
\begin{itemize}
\item[$(i)$] Suppose $G\ncong T(l,k-l,n-k)$ for any $1\leq l\leq k-1$.  If $G$ is not a tree, construct a spanning tree $G'$ from $G$ by deleting some edges. If $G$ is a tree, then take $G'$ same as $G$. By Lemma \ref{edge}, $\varepsilon(G)\leq \varepsilon(G').$ While constructing $G'$ from $G$, the number of pendant vertices may increase in $G'$. Suppose $G'$ has more than $k$ pendant vertices. Since $k\geq 2$, $G'$ has at least one vertex of degree greater than $2$ and at least two paths attached to it. Consider a vertex $v$ of $G'$ with $deg(v)\geq 3$ and two paths $P_{l_1},P_{l_2},\; l_1\geq l_2$ attached at $v.$ Using grafting of edges operation on $G'$, we get a new tree $\tilde{G}$ with number of pendant vertices one less than the number of pendant vertices of $G'$ and by Lemma \ref{Grafting}, $\varepsilon(G')<\varepsilon(\tilde{G}).$ Continue this process till we get a tree with $k$ pendant vertices from $\tilde{G}.$ By Lemma \ref{Grafting}, in every step of this process the total eccentricity index will increase. So, we will reach at a tree  $T$ of order $n$ with $k$ pendant vertices. If $G'$ has $k$ pendant vertices then take $T$ as  $G'.$ Thus $\varepsilon(G')\leq \varepsilon(T)$. By Proposition \ref{k-maxtree}, $\varepsilon(T)\leq \varepsilon(T(l,k-l,n-k))$ for any $1\leq l \leq k-1$. Now the result follows as $\varepsilon(T(l,k-l,n-k))=\left\lfloor\dfrac{3n^2-k^2- 2nk+2(n+k)}{4}\right\rfloor$ for any $1\leq l\leq k-1$.

\item[$(ii)$] Suppose $G$ is not isomorphic to $U_{n,3}^l.$ Since $G$ is connected and has exactly one pendent vertex, it must contain a cycle. Let $C_g$ be a cycle in $G.$ If $G$ is a unicyclic graph then by Proposition \ref{Uc-max}, $\varepsilon(G)\leq \lfloor \frac{3n^2-4n-3}{4} \rfloor $ with equality if and only if $G=U_{n,3}^l$. If $G$ has more than one cycle, then construct a new graph $G'$ from $G$ by deleting edges from all cycles other than $C_g$, so that the graph remains connected and $G'\ncong U_{n,3}^l.$ Then, by Lemma \ref{edge}, $\varepsilon(G)\leq \varepsilon(G')$ and $G'$ is a unicyclic graph on $n$ vertices with girth $g.$ By Proposition \ref{Uc-max}, $\varepsilon(G)\leq \varepsilon(G')\leq \lfloor \frac{3n^2-4n-3}{4} \rfloor $ and equality holds if and only if $G\cong U_{n,3}^l$.  

\item[$(iii)$] Let $G\ncong C_{3,3}^n$. First suppose $G$ has no cut vertex. Then $G$ has exactly one block, which is $G$ itself. By Lemma \ref{Block-transform}, $\varepsilon(G)\leq \varepsilon(C_n)$ and the result follows from Lemma \ref{Compare}.

Now suppose $G$ has at least one cut vertex. Then $G$ has at least two blocks. Let $H_1$ and $H_2$ be two blocks such that distance between them is maximum among all pair of blocks in $G.$ Suppose $|V(H_1)|=n_1$ and $|V(H_2)|=n_2$.

\noindent First suppose $d(H_1,H_2)=0.$ Then there is exactly one cut vertex $w$ in $G$ and every block is a pendant block with at least $3$ vertices. Replace each block by a cycle on same number of vertices to get a new graph $G'$. Then by Lemma \ref{Block-transform}, $\varepsilon(G)\leq \varepsilon(G')$. If there are exactly two cycles in $G'$, then the result follows by Lemma \ref{Compare1}. If there are more than two cycles in $G',$ then keep two cycles say $C_{n_1}$ and $C_{n_2}$ unchanged and from all other cycles delete an edge with an end point $w$ to get a new graph $G''$. Clearly $\varepsilon(G'')\geq \varepsilon(G)$ but number of pendant vertices in $G''$ is more than that of $G$. If there are more than one path attached at $w$ in $G''$, then sequentially apply grafting an edge operation to $G''$ to obtain a new graph $\tilde{G}$ such that $\tilde{G}$ has exactly one path attached at $w$, otherwise take $\tilde{G}$ as $G''$. By Lemma \ref{Grafting}, $\varepsilon(\tilde{G})\geq\varepsilon(G'')$. Note that $\tilde{G}$ is the graph having one pendant vertex, obtained by identifying a vertex $x$ of $C_{n_1}$, a vertex $y$ of $C_{n_2}$ and a pendant vertex $z$ of the path $P_{n+2-n_1-n_2}.$ Let $v$ be the pendant vertex of $\tilde{G}$. Construct a new graph $\bar{G}$ from $\tilde{G}$ by identifying $x$ with $v$ and $y$ with $z$. Then $\bar{G}$ has zero pendant vertex and by Lemma \ref{Path-movement}, $\varepsilon(\bar{G})>\varepsilon(\tilde{G})$. Now the result follows from Lemma $\ref{g-3change}.$

\noindent Now suppose $d(H_1,H_2)\geq 1$. Replace the blocks (if required) $H_1$ and $H_2$ by two cycles $C_{n_1}$ and $C_{n_2}$ respectively to form a new graph $G'$ from $G$. Clearly $G'\in \mathfrak{H}_{n,0}$ and by Lemma \ref{Block-transform}, $\varepsilon(G')\geq \varepsilon(G).$ If $G'$ is isomorphic to $C_{n_1,n_2}^n$ then the result follows from Lemma \ref{g-3change}. Otherwise, there must be some blocks in $G'$ other than $C_{n_1}$ and $C_{n_2}$ with at least three vertices. Remove edges say $\{e'_1,\ldots, e'_p\}$ from all such blocks  to form a new graph $G''$ such that $G''$ is connected and there are no cycles other than  $C_{n_1}$ and $C_{n_2}$. We can choose the edges $\{e'_1,\ldots, e'_p\}$ such that  $G''$ is not isomorphic to $C_{n_1,n_2}^n$. Then by Lemma \ref{edge}, $\varepsilon(G'')\geq \varepsilon(G').$

Let $P: v_1v_2\ldots v_k$ be the path in $G''$ joining $C_{n_1}$ and $C_{n_2}$ where $v_1\in V(C_{n_1})$ and $v_k\in V(C_{n_2})$. Let $v_0$ and $v_{0}'$ be the two vertices on $C_{n_1}$ adjacent with $v_1$, and let $v_{k+1}$ and $v_{k+1}'$ be the two vertices on $C_{n_2}$ adjacent with $v_k$. Consider the edges $e_i=\{v_i,v_{i+1}\}$ for $i=1,2,\ldots,k-1$,$e_0=\{v_1,v_0\}$, $e_0'=\{v_1v_0'\}\}$,$e_{k}=\{v_k,v_{k+1}\}$ and $e_{k}'=\{v_k,v_{k+1}'\}$  in $G''$.

Since $G''\ncong C_{n_1,n_2}^n$ and contains exactly two cycles there are some non trivial trees attached at $v_i$ for some $i=1,2,\ldots,k$. Let $T_i$ be the tree attached at $v_i$ for $i=1,2,\ldots,k.$ Note that for $i=2,\ldots, k-1$, $T_i$  is the component  containing $v_i$ in $G''\setminus \{e_{i-1},e_{1+1}\}$, $T_1$ is the component containing $v_1$  in  $G''\setminus \{ e_1, e_0, e_0'\}$   and $T_k$ is the component containing $v_k$ in $G''\setminus \{e_{k-1},e_{k},e_{k}'\}$. 

Suppose some $T_i$, $i=1,\ldots,k$ are neither trivial trees nor paths with $v_i$ as a pendant vertex. Then form $\tilde{G}$ from $G''$ by sequentially applying grafting of edge operations on those trees such that all $T_i$, $i=1,\ldots,k$ become paths with $v_i$ as a pendant vertex. If all $T_i$, $i=1,\ldots,k$ are either trivial or  paths with $v_i$ as a pendant vertex then take $\tilde{G}$ as $G''$. By Lemma \ref{Grafting}, $\varepsilon(\tilde{G})\geq \varepsilon(G'')$.

Let $v_i$ be the vertex nearest to $v_1$ in $\tilde{T}$ such that $deg(v_i)\geq 3$ (strict inequality occurs only when $v_i=v_1$ or $v_k$, in these cases $d(v_i)=4$) and suppose $w_i$ is the pendant vertex of the path $P_i$ attached at $v_i$. If $v_i=v_1$, delete the edge  $\{v_1,v_2\}$ and add the edge $\{w_1,v_2\}$ and if $v_i\neq v_1$, then delete the edge $\{v_i,v_{i-1}\}$ and add the edge $\{v_{i-1},w_i\}$.  Repeat this till $\deg(v_1)=\deg(v_k)=3$ and $\deg(v_i)=2$ for $i=2,\ldots,k-2.$ . This way we get the graph $C_{n_1,n_2}^n$ from  $\tilde{G}$ and by Lemma \ref{Path-movement}, $\varepsilon(C_{n_1,n_2})>\varepsilon(\tilde{G})$. Now the result follows from Lemma \ref{g-3change}.
\end{itemize} 
\end{proof}

\noindent For $3\leq n\leq 6$, it can be easily checked that $C_n$ uniquely maximizes the total eccentricity index over $ \mathfrak{H}_{n,0}$.
 
\section{Graphs with fixed number of cut vertices}

We denote the set of all connected graphs on $n$ vertices and $s$ cut vertices by $\mathfrak{C_{n,s}}$. Clearly $0\leq s\leq n-2.$ Let $\mathfrak{C_{n,s}^t}$ be the set of all trees on $n$ vertices and $s$ cut vertices. In a tree every vertex is either a cut vertex or a pendant vertex, so  $\mathfrak{C_{n,s}^t}=\mathfrak{T}_{n,n-s}$. The next result follows from Proposition \ref{k-maxtree}.
 
\begin{theorem}
Let $T\in \mathfrak{C_{n,s}^t}.$ Then
$$\varepsilon(T)\leq \varepsilon(T(l,n-l-s,s)),$$ for any $1\leq l \leq n-s-1,$ and
$$\varepsilon(T)\geq \begin{cases}
\varepsilon(T_{n,n-s}) &\mbox{if}\;\; (n-s)\nmid n-2,\\
\varepsilon(T_{n,n-s}^t) &\mbox{if}\;\; (n-s)\mid n-2.
\end{cases}$$
for any $1\leq t \leq n-s-1.$
\end{theorem}

For $2\leq m\leq n,$ let $v_1,v_2,\ldots,v_m$ be the vertices of a complete graph $K_m$. For $i=1,2,\ldots,m$ consider the paths $P_{l_i}$, $l_i\geq 1$ such that $l_1+l_2+\cdots+l_m=n$. By identifying a pendant vertex of the path $P_{l_i}$  with the vertex $v_i$  (if $l_i =1$, then identify the single vertex with $v_i$), for $i=1,2,\ldots,m$, we obtain a graph on $n$ vertices with $n-m$ cut vertices. We denote this graph by $K_m^n(l_1,l_2,\ldots, l_m)$. 
\begin{lemma}\label{Path-balance}
Let $m\geq 2$ and  $l_k=\max\{l_1,l_2,\ldots,l_m\}$. Suppose  $l_j\leq l_k-2,$ for some $j\in \{1,\ldots,l_{k-1},l_{k+1},\ldots,l_m\}$. Then $\varepsilon(K_m^n(l_1,l_2,\ldots,l_j+1,\ldots,l_k-1,\ldots,l_m))\leq \varepsilon(K_m^n(l_1,l_2,\ldots, l_m)).$
\end{lemma}
\begin{proof}
Let $G\cong K_m^n(l_1,l_2,\ldots, l_m)$. Let $P_{l_j}:u_1u_2\cdots u_{l_j}$ and $P_{l_k}:w_1w_2\cdots w_{l_k}$  be the paths in $G$ such that $\deg(u_1)=\deg(w_1)=1$,$u_{l_j}=v_j$ and $w_{l_k}=v_k$. Delete the edge $\{w_1,w_2\}$ and add the edge $\{u_1,w_1\}$ in $G$ to form a new graph $G'$. Clearly $G'\cong K_m^n(l_1,l_2,\ldots,l_j+1,\ldots,l_k-1,\ldots,l_m)$. 

If $m=2$ then $G\cong P_{l_k+l_j}\cong G'$ and so $\varepsilon(G')= \varepsilon(G).$ Therefore, assume $m\geq 3.$
Let $S_1=\{w_1,w_2,\ldots, w_{l_k}=v_k, u_1,\ldots,u_{l_j}=v_j\}$ and $S_2=V(G)\setminus S_1= V(G')\setminus S_1.$ So, $S_1\cap S_2=\emptyset$ and $S_1\cup S_2=V(G)=V(G')$. We show that, while moving to $G'$ from $G$, the eccentricity of each vertex either decreases or remains the same.

Suppose $v\in V(G)=V(G')$. Following are the two cases.
 
\noindent {\bf Case-I:} $v\in S_2$

Since $l_k=\max\{l_1,l_2,\ldots,l_m\}$, for any $v\in S_2$,  $w_1$ is an eccentric vertex of $v$ in $G$. So, $e_G(v)=d_G(v,w_1)$ and $e_{G'}(v)$ is either $d_{G'}(v,w_2)=d_{G}(v,w_2)$ or $d_{G'}(v,w_2)+1$. Hence $e_G(v)\geq e_{G'}(v)$ for  $v\in S_2.$

\noindent {\bf Case-II:} $v\in S_1$

Let $l_q=\max\{l_1,\ldots,l_{k-1},l_{k+1},\ldots,l_m\}$ and let $z$ be the pendant vertex of $G$ associated with $P_{l_q}$. Since $l_k\geq l_j+2$ so $l_q\geq l_j.$ First consider  $l_q>l_j$. Then $e_G(v)=d_G(v,w_1)$ or $d_G(v,z)$.

\underline{Subcase-I:} $e_G(v)=d_G(v,w_1)$
Then $v\neq w_1$ and  $d_G(v,z)\leq d_G(v,w_1)$. Since $l_q>l_j$, so $e_{G'}(v)=d_{G'}(v,w_2)$ or $d_{G'}(v,z)$.  But $d_{G'}(v,w_2)=d_{G}(v,w_2)<d_{G}(v,w_1)=e_G(v)$ and $d_{G'}(v,z)=d_{G}(v,z)\leq d_G(v,w_1)=e_G(v)$. So, $e_{G'}(v)\leq e_G(v)$.
 
\underline{Subcase-II:} $e_G(v)=d_G(v,z).$ 

If $v=w_1$, then $d_{G'}(w_1,w_2)=d_G(w_1,u_1)<d_G(w_1,z).$ We have $e_{G'}(w_1)=d_{G'}(w_1,w_2)<d_G(w_1,z)=e_G(w_1)$ or $e_{G'}(w_1)=d_{G'}(w_1,z)=l_j+l_q<l_k-1+l_q=d_G(w_1,z)=e_G(w_1)$.

If $v\neq w_1$ then a similar argument as in Subcase-I will give $e_{G'}(v)\leq e_G(v)$.

Now consider $l_q=l_j.$ Then  either $w_1$ or $u_1$ is an eccentric vertex of $v$ in $G$. Also in $G'$ either $w_2$ or $w_1$ is an eccentric vertex of $v$. Let $A_1$ be the $w_1-u_1$ path in $G$ and let $A_2$ be the $w_2-w_1$ path in $G'$. Then $|V(A_1)|=|V(A_2)|=l_k+l_j.$ So $\sum_{v\in S_1} e_G(v)=\varepsilon(P_{l_k+l_j})=\sum_{v\in S_1} e_{G'}(v)$. 

Hence  from Case-I and Case-II, $\varepsilon(G')\leq \varepsilon(G)$ and this completes the proof. 
\end{proof}
\begin{corollary}\label{Path-balance1}
Let $G\cong K_{n-s}^n(l_1,l_2,\ldots,l_{n-s})$ for some $l_1,l_2,\ldots,l_{n-s}$. Then\\ 
$\varepsilon(K_{n-s}^n(l_1',l_2',\ldots,l'_{n-s}))\leq \varepsilon(G)$, where $|l'_i-l'_j|\leq 1$ for every  $i,j\in \{1,2,\ldots,n-s\}$. 
\end{corollary}
  
We now count $\varepsilon(K_{n-s}^n(l_1,l_2,\ldots,l_{n-s}))$ where $|l_i-l_j|\leq 1$ for every $i,j\in \{1,2,\ldots,n-s\}$. Let $q=\lfloor\frac{n}{n-s}\rfloor$ and take $r=n-(n-s)q$. Then $0 \leq r< n-s$. Observe that $K_{n-s}$ is a subgraph of $K_{n-s}^n(l_1,l_2,\ldots,l_{n-s})$ and $K_{n-s}^n(l_1,l_2,\ldots,l_{n-s})\setminus E(K_{n-s})\cong rP_{q+1}\cup (n-s-r)P_q$. We may consider $P_q$ and $P_{q+1}$ as $P_q:u_1u_2\cdots u_q$ and $P_{q+1}:w_1w_2\cdots w_{q+1}$, respectively, where $u_1$ and $w_1$ are pendant vertices in $K_{n-s}^n(l_1,l_2,\ldots,l_{n-s})$.\\

\noindent If $r=0$, then $q=\frac{n}{n-s}$ and $K_{n-s}^n(l_1,l_2,\ldots,l_{n-s})\setminus E(K_{n-s})\cong (n-s)P_q$. So, 
\begin{align*}
 \varepsilon(K_{n-s}^n(l_1,l_2,\ldots,l_{n-s}))&=(n-s) \sum_{i=1}^{q}{e_{K_{n-s}^n(l_1,l_2,\ldots,l_{n-s})}(u_i)}\\
 &=(n-s)\sum_{i=1}^q{(2q-i)}\\
 &=q(3q-1)\left(\frac{n-s}{2}\right)\\
 &=\left(\frac{n}{n-s}\right)\left(\frac{3n}{n-s}-1\right)\left(\frac{n-s}{2}\right)\\
 &=\frac{n(2n+s)}{2(n-s)}.
\end{align*}

\noindent If $r=1$ then $K_{n-s}^n(l_1,l_2,\ldots,l_{n-s})\setminus E(K_{n-s})\cong P_{q+1}\cup (n-s-1)P_q$. So

\begin{align*}
\varepsilon(K_{n-s}^n(l_1,l_2,\ldots,l_{n-s}))&=\sum_{j=1}^{q+1}{e_{K_{n-s}^n(l_1,l_2,\ldots,l_{n-s})}(w_j)}+(n-s-1) \sum_{i=1}^{q}{e_{K_{n-s}^n(l_1,l_2,\ldots,l_{n-s})}(u_i)}\\
 &=\sum_{j=1}^{q+1}{(2q+1-j)}+(n-s-1)\sum_{i=1}^q{(2q+1-i)}\\
 &=(n-s)\sum_{i=1}^q{(2q+1-i)}+q\\
 &=\frac{q}{2}\left[3(n-s)q+(n-s+2)\right].
\end{align*}
  
\noindent If $r\geq 2$ then $K_{n-s}^n(l_1,l_2,\ldots,l_{n-s})\setminus E(K_{n-s})\cong rP_{q+1}\cup (n-s-r)P_q$.  So

\begin{align*}
 \varepsilon(K_{n-s}^n(l_1,l_2,\ldots,l_{n-s}))&=r\sum_{j=1}^{q+1}{e_{K_{n-s}^n(l_1,l_2,\ldots,l_{n-s})}(w_j)}+(n-s-r) \sum_{i=1}^{q}{e_{K_{n-s}^n(l_1,l_2,\ldots,l_{n-s})}(u_i)}\\
 &=r\sum_{j=1}^{q+1}{(2q+2-j)}+(n-s-r)\sum_{i=1}^q{(2q+1-i)}\\
 &=(n-s)\sum_{i=1}^q{(2q+1-i)}+2rq+r\\
 &=\frac{1}{2}\left[2r(2q+1)+q(3q+1)(n-s)\right].
\end{align*}
This leads to the following Lemma.
\begin{lemma}
 Let $0\leq s\leq n-2$. Then $$\varepsilon(K_{n-s}^n(l_1,l_2,\ldots,l_{n-s}))=
 \begin{cases}
 \frac{n(2n+s)}{2(n-s)} & \mbox{if $r=0$},\\
 \frac{q}{2}\left(3(n-s)q+(n-s+2)\right) &\mbox{if $r=1$},\\
 \frac{1}{2}\left[2r(2q+1)+q(3q+1)(n-s)\right] &\mbox{if $r\geq 2$},
 \end{cases}$$
 
 where $q=\lfloor\frac{n}{n-s}\rfloor$ and $r=n-(n-s)q.$
\end{lemma}

We will now prove the main result of this section regarding minimization . 
\begin{theorem}\label{mincut-thm}
 Let $0\leq s\leq n-2$ and $G\in \mathfrak{C}_{n,s}$. Then $\varepsilon(K_{n-s}^n(l_1',l_2',\ldots,l'_{n-s}))\leq \varepsilon(G)$, where $|l'_i-l'_j|\leq 1$ for every $i,j\in \{1,2,\ldots,n-s\}$.
\end{theorem}

\begin{proof} 
Let $G\in  \mathfrak{C}_{n,s}$ and let $H$ be the graph $K_{n-s}^n(l'_1,l'_2,\ldots,l'_{n-s})$ where $|l'_i-l'_j|\leq 1$ for every $i,j\in \{1,2,\ldots,n-s\}$.  If $G\cong K_{n-s}^n(l_1,l_2,\ldots,l_{n-s})$ for some $l_1,l_2,\ldots,l_{n-s}$ then by Corollary \ref{Path-balance1}, $\varepsilon(H)\leq \varepsilon(G)$.

Suppose $G$ is not isomorphic to $K_{n-s}^n(l_1,l_2,\ldots,l_{n-s})$ for any $l_1,l_2,\ldots,l_{n-s}$. If some blocks, say $B_1,B_2,\ldots,B_l$ of $G$ are not complete, then form a new graph $G_1$ from $G$ by joining the non-adjacent vertices of each $B_i,\;1\leq i\leq l$ with edges such that each block of $G_1$ becomes complete, otherwise take $G_1$ as $G$. Observe that $G_1\in \mathfrak{C}_{n,s}$ and by Lemma \ref{edge}, $\varepsilon(G_1)\leq \varepsilon (G).$ If $s=0,$ then $G\cong K_n\cong K_n^n(1,1,\ldots,1)$ and the result follows.

Suppose $s\geq 1$. Then every cut vertex of $G_1$ is shared by at least two blocks. If $w$ is a cut vertex of $G_1$ and $C_1,C_2,\ldots, C_{k}$ with $k\geq 3$ are the blocks sharing the vertex $w$, then  join every pair of non adjacent vertices of $\bigcup_{i=2}^k{V(C_i)}$ by an edge. Repeat this for each cut vertex of $G_1$ which is shared by more than two blocks. This way we get a new graph $G_2$. If every cut vertex of $G_1$ is shared by exactly two blocks then take $G_2$ as $G_1.$ Clearly  $G_2\in \mathfrak{C}_{n,s}$ and by Lemma \ref{edge}, $\varepsilon(G_2)\leq \varepsilon (G_1)$. Note that $G_2$ is a graph in which, every block is complete and every cut vertex is shared by exactly two blocks. If $G_2\cong K_{n-s}^n(l_1,l_2,\ldots,l_{n-s})$, for some $l_1,l_2,\ldots,l_{n-s}$ then by Corollary \ref{Path-balance1}, $\varepsilon(H)\leq \varepsilon(G_2)$ and the result follows. 

Suppose $G_2$ is not isomorphic to $K_{n-s}^n(l_1,l_2,\ldots,l_{n-s})$ for any $l_1,l_2,\ldots,l_{n-s}$. We further consider two cases depending on whether $s=1$ or $s\geq 2$ .

First suppose $s=1$. Then $G_2$ has exactly two complete blocks with a common cut vertex $w$. Let $B_1$ and $B_2$ be the  two blocks of $G_2$ with   $V(B_1)=\{u_1,u_2,\ldots,u_{m_1}=w\}$ and $V(B_2)=\{v_1,v_2,\ldots,v_{m_2}=w\}$ with $m_1,m_2 \geq 3.$ Construct a new graph $G_2'$ from $G_2$ as follow: Delete the edges $\{u_1,u_i\} ,i=2,3,\ldots,m_1-1$ and add the edges $\{u_i,v_j\},i=2,3,\ldots,m_1-1;j=1,2,\ldots,m_2-1.$ Then $V(G_2')=V(G_2)$ and $G_2'$ is isomorphic to $ K_{n-1}^n(2,1,\ldots,1)$ . Note that $e_{G_2'}(w)=e_{G_2}(w)=1$ and for $x\in V(G_2')\setminus \{ w \}=V(G_2)\setminus \{ w \}$, $e_{G_2'}(x)=e_{G_2}(x)=2$. So $\varepsilon(G_2')=\varepsilon(G_2)$ and the result follows.

Now suppose $s\geq 2.$ Then $G_2$ has $s+1$ blocks and  $B_{G_2}$, the block graph of $G_2$ is a tree. So $G_2$ has either one central block or two adjacent central blocks and at least one non central block. Let $C$ be a central block in $G_2$. Suppose  $P$ is a longest path in $G_2.$ Then $P$ passes through exactly two vertices of $C$. Suppose $B_1$ is a non central block in $G_2$  which is not isomorphic to $K_2$. Then $|V(B_1)|\geq 3$. Let  $b\in V(B_1)$ be a cut vertex of $G_2$ such that $d(B_1, C)=d(b,c)$, where $c\in V(C)$ is a cut vertex of $G_2$. Let $B_2$ be the block adjacent to $B_1$ sharing the cut vertex $b$ with $B_1$ ($B_2$ may be same as $C$). Let $V(B_1)=\{b=v_1,v_2,\ldots v_{m_1}\}$ and $V(B_2)=\{b=u_1,u_2,\ldots, u_{m_2}\}$. If $P$ passes through $B_1,$ then $P$ also passes through $B_2$. So, $P$ must contain $b$ and some other vertex of $B_1$, say $v_2$. Construct a new graph $G_2'$  from $G_2$ by deleting the edges $\{v_2,v_i\}\; i=3,4,\ldots m_1$ and adding the edges $\{v_i,u_j\}\; i=3,4,\ldots,m_1, j=2,3,\ldots,m_2$. If $P$ does not pass through $B_1$, we can chose any vertex of $B_1$ in place of $v_2$. Clearly $G_2'\in  \mathfrak{C}_{n,s}$ and the number of blocks in $G_2'$ is same as the number of blocks in $G_2.$ Note that $P$ is still a longest path in $G_2'$. So, the block corresponding to $C$ in  $G_2'$ is still a central block in $G_2'$. We will now show that $\varepsilon(G_2')\leq \varepsilon(G_2)$.

Let $v\in V(G_2)=V(G_2')$ and let $v'$ be an eccentric vertex of $v$ in $G_2$ lies on the longest path $P$. Since $P$ is  a longest path in both $G_2$ and $G_2'$, so $v'$ is also an eccentric vertex of $v$ in $G_2'$. Thus we have $d_{G_2'}(v,v')\leq d_{G_2}(v,v')$. So,  $e_{G_2'}(v)= d_{G_2'}(v,v')\leq d_{G_2}(v,v')=e_{G_2}(v)$ and hence $\varepsilon(G_2')\leq \varepsilon(G_2)$.

If $G_2'$ has some non central block which is not isomorphic to $K_2$, then repeat the same process until we get a graph in which all the non central blocks are $K_2$. Name the new graph as $\tilde{G}$. Note that (as we have shown while moving to $G_2'$ from $G_2$) in each intermediate step between $G_2$ and $\tilde{G},$ the  block corresponding to $C$ is the central block in each step and the total eccentricity index decreases or remains the same.  If all the non central blocks of $G_2$ are isomorphic to  $K_2$ then take $G_2$ as $\tilde{G}$. So $\varepsilon(\tilde{G})\leq \varepsilon(G_2).$

If $\tilde{G}$ has exactly one central block then it is isomorphic to $K_{n-s}^n(l_1,l_2,\ldots,l_{n-s})$ for some positive integers $l_1,l_2,\ldots,l_{n-s}$ and the result follows from Corollary \ref{Path-balance1}. 

Suppose $\tilde{G}$ has two adjacent central blocks $C$ and $C'$ sharing the cut vertex $z$. Let $V(C)=\{x_1,x_2,\ldots,x_s=z\}$ and $V(C')=\{y_1,y_2,\ldots,y_t=z\}$. Since $P$ is a longest path of $\tilde{G}$, it must contain a vertex , say $x_1$ of $C$ different from $z$. Construct a new graph $\bar{G}$ from $\tilde{G}$ by deleting the edges $\{x_1,x_i\},\; 2\leq i\leq s-1$ and adding the edges $\{x_i,y_j\}\; 2\leq i \leq s-1; 1\leq j\leq t-1$. For $v\in V(\tilde{G})=V(\bar{G})$, the shortest path between $v$ and its eccentric vertex( both for $\tilde{G}$ and $\bar{G}$) passes through $z$, so $e_{\tilde{G}}(v)=e_{\bar{G}}(v).$ Hence $\varepsilon(\bar{G})= \varepsilon ({\tilde{G_2}})$. Note that $\bar{G}$ is isomorphic to $ K_{n-s}^n(l_1,l_2,\ldots,l_{n-s})$ for some $l_1,l_2,\ldots,l_{n-s}$. So the result follows from Corollary \ref{Path-balance1}.
\end{proof}

We will now study about the graphs which maximize the total eccentricity index over $\mathfrak{C}_{n,s}$. We have the following theorem for $s=0.$
\begin{theorem}
Let $n\geq 3$ and $G\in \mathfrak{C}_{n,0}$. Then $\varepsilon(G)\leq n\lfloor\frac{n}{2}\rfloor$ and equality happens if $G\cong C_n.$
\end{theorem}
\begin{proof}
Since $s=0$, so $G$ has exactly one block and the result follows from Lemma \ref{Block-transform}.
\end{proof}
We will now find a graph which maximizes the total eccentricity index over $\mathfrak{C}_{n,1}$. To obtain that graph, the following lemma is very useful.
\begin{lemma}\label{C-reduce}
Let $H$ be a graph with at least two vertices and $w\in V(H)$. Let $G$ be the graph obtained from $H$, $C_{m_1}$ and $C_{m_2}$ by identifying $w$, a vertex of $C_{m_1}$ and a vertex of $C_{m_2}$. Let  $G'$ be the graph obtained from $H$ and $C_{m_1+m_2-1}$ by identifying $w$ with a vertex of $C_{m_1+m_2-1}$. Then $\varepsilon(G')\geq \varepsilon(G)$.
\end{lemma}
\begin{proof}
Consider the following labelling  of vertices of $C_{m_1}$, $C_{m_2}$ and $C_{m_1+m_2-1}$:
$$C_{m_1}:wu_1 \cdots u_{\lfloor \frac{m_1}{2} \rfloor} u_{\lfloor \frac{m_1}{2} \rfloor+1} \cdots u_{m_1-1}w,$$ 
$$C_{m_2}:w v_1 \cdots v_{\lfloor \frac{m_2}{2} \rfloor} v_{\lfloor \frac{m_2}{2} \rfloor+1} \cdots v_{m_2-1} w $$ 
and $$C_{m_1+m_2-1}:w u_1 \cdots  u_{m_1-1} v_{m_2-1} v_{m_2-2} \cdots  v_1 w.$$ 

\noindent Let $v\in V(G)=V(G')$ and $v'$ be an eccentric vertex of $v$ in $G$. Let $h\in V(H)$ such that $d(h,w)=\max\{d(x,w):x \in V(H)\}$. Without loss of generality, assume that $m_1\geq m_2$.  We now consider the following cases.\\

\noindent{\bf Case I:} $d(w,h)>\lfloor \frac{m_1}{2} \rfloor.$ \\
Suppose $v\in V(G)\setminus V(H)=\{u_1,u_2,\ldots,u_{m_1-1},v_1,v_2,\ldots,v_{m_2-1}\}.$ Then we can take $v'$ as $h$ and so $d_{G'}(v,h)=d_{G'}(v,w)+d_{G'}(w,h)\geq d_G(v,w)+d_G(w,h)=d_G(v,h)=e_G(v)$. This implies $e_{G'}(v)\geq e_G(v)$. \\
Now suppose $v\in V(H)$. Then either $v'\in V(H)$ or $v'=u_{\lfloor \frac{m_1}{2} \rfloor}$. If  $v'\in V(H)$, then $d_{G'}(v,v')=d_{G}(v,v')=e_G(v).$ This implies $e_{G'}(v)\geq e_G(v).$ If $v'=u_{\lfloor \frac{m_1}{2} \rfloor},$ then $d_{G'}(v,v')=d_{G'}(v,u_{\lfloor \frac{m_1}{2} \rfloor})=d_{G'}(v,w)+d_{G'}(w,u_{\lfloor \frac{m_1}{2} \rfloor})=d_G(v,w)+d_G(w,u_{\lfloor \frac{m_1}{2} \rfloor})=d_G(v,v')=e_G(v)$. This implies $e_{G'}(v)\geq e_G(v).$ Thus $\varepsilon(G')\geq \varepsilon(G).$ 

\noindent{\bf Case II:} $d(w,h)\leq \lfloor \frac{m_1}{2} \rfloor.$ \\
Suppose $v\in V(H)$. For any $z\in V(H)$, $d_G(v,z)\leq d_G(v,w)+d_G(w,z)\leq d_G(v,w)+\lfloor \frac{m_1}{2} \rfloor=d_G(v,u_{\lfloor \frac{m_1}{2} \rfloor})$. Therefore, we can take $v'$ as $u_{\lfloor \frac{m_1}{2} \rfloor}$ and so the eccentric vertices of $v$ in $G'$ are in $C_{m_1+m_2-1}$. Therefore, $e_{G'}(v)=d_{G'}(v,w)+\lfloor \frac{m_1+m_2-1}{2} \rfloor  >d_G(v,w)+\lfloor \frac{m_1}{2} \rfloor=e_G(v).$ This implies $e_{G'}(v)\geq e_G(v)+1$ for all $v\in V(H).$

 Now suppose $v\in V(G)\setminus V(H)$.  If $v'\in V(H)$, we can choose $v'=h$ and in this case  $d_{G'}(v,v')=d_{G'}(v,w)+d_{G'}(w,h)\geq d_G(v,w)+d_G(w,v')=d_G(v,v')=e_G(v)$. This implies $e_{G'}(v)\geq e_G(v).$  If  $v' \in V(G)\setminus V(H)$ then we have two subcases.

\noindent{\bf Subcase I:} At least one of  $m_1$ or $m_2$ is odd.\\
In this case, $e_G(v)\leq \lfloor \frac{m_1}{2} \rfloor+\lfloor \frac{m_2}{2} \rfloor \leq \lfloor \frac{m_1+m_2-1}{2} \rfloor \leq e_{G'}(v).$\\

\noindent{\bf Subcase II:} Both $m_1$ and $m_2$ are even.\\ 
Suppose  $v\in \{u_1,u_2,\ldots,u_{m_1-1},v_1,v_2,\ldots,v_{m_2-1}\}\setminus \{u_{\frac{m_1}{2}},v_{\frac{m_2}{2}}\}$.  Then $d_G(v,v')\leq \frac{m_1}{2}+\frac{m_2}{2}-1$ and $e_{G'}(v)\geq \lfloor\frac{m_1+m_2-1}{2}\rfloor=\frac{m_1+m_2-2}{2}=\frac{m_1}{2}+\frac{m_2}{2}-1\geq d_G(v,v')=e_G(v).$

For $v\in \{u_{\frac{m_1}{2}},v_{\frac{m_2}{2}}\}$,  $e_G(v)=  \frac{m_1+m_2}{2}$. We have 
$$e_{G'}(u_{\frac{m_1}{2}})\geq\left\lfloor\frac{m_1+m_2-1}{2}\right\rfloor= \frac{m_1+m_2}{2}-1=e_G(u_{\frac{m_1}{2}})-1$$
and  
$$e_{G'}(v_{\frac{m_2}{2}})\geq\left\lfloor\frac{m_1+m_2-1}{2}\right\rfloor=\frac{m_1+m_2}{2}-1=e_G(v_{\frac{m_2}{2}})-1.$$
As $|V(H)|\geq 2,$ there exists $w'\in V(H)$ different from $w$ such that 
\begin{align*}
e_{G'}(u_{\frac{m_1}{2}})+e_{G'}(v_{\frac{m_2}{2}})+e_{G'}(w)+e_{G'}(w')& \geq e_{G}(u_{\frac{m_1}{2}})-1+e_{G}(v_{\frac{m_2}{2}})-1+e_{G}(w)+1+e_{G}(w')+1\\&=e_{G}(u_{\frac{m_1}{2}})+e_{G}(v_{\frac{m_2}{2}})+e_{G}(w)+e_{G}(w').
\end{align*}

Therefore, $\varepsilon(G')\geq \varepsilon(G)$ and  this completes the proof.
\end{proof}

Now we count the total eccentricity index of the graph $C_{m_1,m_2}^n$, where $m_1+m_2-1=n.$ There are four cases depending upon whether $m_1$ and $m_2$ are even or odd. Let us consider the case when $m_1$ and $m_2$ are both even. Let us label the vertices of $C_{m_1}$ and $C_{m_2}$ in $C_{m_1,m_2}^n$ as $C_{m_1}: wu_1\cdots u_{\frac{m_1}{2}-1} u_{\frac{m_1}{2}} u_{\frac{m_1}{2}+1} \cdots u_{m_1-1}w$ and $C_{m_2}:wv_1v_2\cdots v_{\frac{m_2}{2}-1} v_{\frac{m_2}{2}} v_{\frac{m_2}{2}+1} \cdots v_{m_2-1}w$. Without loss of generality, assume that $m_1\geq m_2.$ Take $m_1=m_2+k$, so $k\geq 0$ is even. Then 
 \begin{itemize}
 
 \item for $i=1,2,\ldots \frac{m_2}{2}$, $e(v_i)=e(v_{m_2-i})=i+\frac{m_1}{2}$;  
 \item $e(w)=\frac{m_1}{2}$;
 \item for $j=1,2,\ldots \frac{k}{2}$, $e(u_j)=e(m_1-j)=\frac{m_1}{2}$;  
 \item for $\frac{k}{2}+1\leq j\leq \frac{m_1}{2}$, $e(u_{j})=e(u_{m_1-j})=j+\frac{m_2}{2}.$ 
 
\end{itemize} 
So,
 \begin{align*}
 \varepsilon(C_{m_1,m_2}^n)&= e(w) + \sum_{i=1}^{m_2-1}{e(v_i)}+ \sum_{j=1}^{m_1-1}{e(u_j)}\\
 &= e(w)+2\sum_{i=1}^{\frac{m_2}{2}-1}{e(v_i)}+e(v_{\frac{m_2}{2}}) + 2\sum_{j=1}^{\frac{k}{2}}{e(u_j)}+2\sum_{j=\frac{k}{2}+1}^{\frac{m_1}{2}-1}{e(u_{j})}+ e(u_{\frac{m_1}{2}})\\
 &=\frac{1}{2}(m_1^2+m_2^2+m_1m_2-m_1).
 \end{align*}
For the other three cases the total eccentricity index of $C_{m_1,m_2}^n$ with $m_1+m_2-1=n$ can be counted similarly. Based on these calculation, we have 
$$\varepsilon(C_{m_1,m_2}^n) = \begin{cases}
 \frac{1}{2}(m_1^2+m_2^2+m_1m_2-m_1) &\mbox{if both $m_1$ and $m_2$ are even},\\
 \frac{1}{2}(m_1^2+m_2^2+m_1m_2-m_1-m_2) &\mbox{if $m_1$ is even and $m_2$ is odd},\\
 \frac{1}{2}(m_1^2+m_2^2+m_1m_2-2m_1+1) &\mbox{if $m_1$ is odd and $m_2$ is even},\\
  \frac{1}{2}(m_1^2+m_2^2+m_1m_2-2m_1-m_2) &\mbox{if both $m_1$ and $m_2$ are odd}.
 \end{cases}$$
 
 Next we count the total eccentricity index of $U_{n,g}^p$ where $3\leq g \leq n-1.$ Suppose $g$ is even. Let $w_1,w_2,\ldots,w_{n-g}$ be the pendant vertices of $U_{n,g}^p$ and $C_g: v_0v_1\cdots v_{g-1}v_0$ such that $deg(v_0)=n-g+2.$ Then we have
\begin{itemize}
\item for $j=1,\cdots,n-g$, $e(w_j)=1+\frac{g}{2}$
\item for $i \in \{0,1,\ldots,g-1\}\setminus \{\frac{g}{2}\}$, $e(v_i)=\frac{g}{2}$
\item $e(v_{\frac{g}{2}})=\frac{g}{2}+1$.
\end{itemize} 
 
 So, $\varepsilon(U_{n, g}^p)=(n-g)(1+\frac{g}{2})+(g-1)(\frac{g}{2})+(\frac{g}{2}+1)=  \frac{ng}{2}+n-g+1$. Similarly $\varepsilon(U_{n, g}^p)$ can be counted when $g$ is odd and we get 
 $$\varepsilon(U_{n, g}^p)=\begin{cases}
 \frac{ng}{2}+n-g+1    &\mbox{if $g$ is even},\\
\frac{n(g-1)}{2}+n-g+2 &\mbox{if $g$ is odd}.
\end{cases}$$   
 
In particular for $g=n-1$, $U_{n,n-1}^p\cong U_{n,n-1}^l$  and 
\begin{equation}\label{eq1}
\varepsilon(U_{n,n-1}^l)=\begin{cases}
\frac{n(n-2)}{2}+3  &\mbox{if $n$ is even},\\
\frac{n(n-1)}{2}+2  &\mbox{if $n$ is odd}.
\end{cases}
\end{equation}
Based on these calculations we have the following lemma.
\begin{lemma}\label{Cut1-compare}
Let $m_1\geq m_2\geq 3$ and $n\geq 5.$ 
\begin{itemize}
\item[$(i)$] If $m_1+m_2-1=n$ then $\varepsilon(C_{m_1,m_2}^n)\leq\varepsilon(U_{n,n-1}^l)$. Furthermore the equality happens if and only if $m_1$ is even and $m_2=3.$
\item[$(ii)$] For $3\leq g\leq n-2$, $\varepsilon(U_{n,g}^p)\leq \varepsilon(U_{n,g+1}^p)$ and equality holds if and only if $g$ is even.
\end{itemize}
\end{lemma}
For $n=3$, the path $P_3$ is the only graph with one cut vertex. For $n=4$, the star $K_{1,3}$ and $U_{4,3}^l$ are the only two graphs with one cut vertex and $\varepsilon(K_{1,3})=\varepsilon(U_{4,3}^l)=7$. So we consider $n\geq 5.$
\begin{theorem}\label{s=1}
Let $n\geq 5$ and $G\in \mathfrak{C}_{n,1}.$ Then $\varepsilon(G) \leq \varepsilon(U_{n,n-1}^l)$.
\end{theorem}
\begin{proof} Suppose $G$ is not isomorphic to $U_{n,n-1}^l$. If $G$ has no cycle then $G\cong K_{1,n-1}$ and $\varepsilon(G)=\varepsilon(K_{1,n-1})=2n-1<\varepsilon(U_{n,n-1}^l)$, by (\ref{eq1}). Suppose $G$ has some cycles. Since $G$ has a unique cut vertex, so all the blocks of $G$ are pendant blocks.  Let $w$ be the  cut vertex in $G$ and let $B_1,B_2,\ldots, B_k$ be the blocks of $G$ with atleast three vertices. Construct a new graph $G'$ from $G$ by replacing each $B_i,\; 1\leq i \leq k$ with a cycle on same number of vertices. Then $G' \in \mathfrak{C}_{n,1}$ and by Lemma \ref{Block-transform}, $\varepsilon(G) \leq \varepsilon(G')$. If $G'$ has exactly one cycle, then $G'$ is isomorphic to $U_{n,g}^p$ for some $g\geq 3$. The result follows from Lemma \ref{Cut1-compare} $(ii)$. 

Suppose $G'$ has at least two cycles. Let $C_{m_1}$ and $C_{m_2}$ be two cycles in $G'$. If $m_1+m_2-1=n,$ then $G'\cong C_{m_1,m_2}^n$. So, by Lemma \ref{Cut1-compare} $(i)$, $\varepsilon(G')\leq \varepsilon(U_{n,n-1}^l)$ and the result follows. If $n>m_1+m_2-1$, then there are at least three blocks sharing the common vertex $w$ in $G'$. Replace the blocks $C_{m_1}$ and $C_{m_2}$  by the cycle $C_{m_1+m_2-1}$ in $G'$ to get a new graph $G''$. Note that $G''\in \mathfrak{C}_{n,1} $ and by Lemma \ref{C-reduce}, $\varepsilon(G')\leq \varepsilon(G'')$. 

If all the blocks of $G''$ are cycles, then repeat this process(if necessary) until we get a graph $\tilde{G}$ on exactly two blocks. By Lemma \ref{C-reduce}, $\varepsilon(G'')\leq \varepsilon(\tilde{G})$ and $\tilde{G}\cong C_{m,m'}^n$ where $m+m'-1=n$. Now the result follows from Lemma \ref{Cut1-compare} $(i)$. 
 
If $G''$ contains $K_2$ as block  then repeat the above process (if necessary) until we get a graph $\bar{G}$ having exactly one cyclic block. Note that $\bar{G}\in \mathfrak{C}_{n,1} $ and $\bar{G}\cong U_{n,g}^p$ for some $g\geq 3.$ By Lemma \ref{Cut1-compare} $(ii)$, $\varepsilon (\bar{G})\leq \varepsilon(U_{n,n-1}^l)$ and this completes the proof.
\end{proof}
 The path $P_n$ is the only graph in $\mathfrak{C}_{n,n-2}.$ We will now obtain a graph which maximizes the total eccentricity index over $\mathfrak{C}_{n,n-3}$.
\begin{theorem}\label{s=n-3}
Let $n\geq 5$ and $G\in \mathfrak{C}_{n,n-3}.$ Then $\varepsilon(G) \leq \varepsilon(U_{n,3}^l).$ 
\end{theorem}
\begin{proof}
Suppose $G$ is not isomorphic to $U_{n,3}^l.$ If $G$  is a tree then it has exactly one vertex of degree $3$. Using grafting of edges operation sequentially  (if necessary), we get  the tree $T(2,1,n-3)$ from $G$ and  by Lemma \ref{Grafting}, $\varepsilon(G) \leq \varepsilon(T(2,1,n-3)).$ Let $v_1$ and $v_2$ be the two pendant vertices of $T(2,1,n-3)$ such that $d(v_1,v_2)=2$. Form a new graph $G'$ from $T(2,1,n-3)$ by joining $v_1$ and $v_2$ with an edge. Then  $G'\cong U_{n,3}^l$. Since $\varepsilon(T(2,1,n-3))=\varepsilon(U_{n,3}^l)$ so the result follows. If $G$ is not a tree then it must be a unicyclic graph with girth $3$ and the result follow from Proposition \ref{Uc-max}.
\end{proof}
We end this section with the following conjecture.
\begin{conjecture}
Let $n\geq 5$ and $2\leq s \leq n-4$. If $G\in \mathfrak{C}_{n,s}$ then $\varepsilon(G) \leq \varepsilon(U_{n,n-s}^l).$ 
\end{conjecture}

\end{document}